\documentclass[12points,reqno]{amsart}
\usepackage{amsfonts}
\usepackage{amssymb}
\usepackage{graphicx}
\usepackage{amsmath}

\newtheorem{lemma}{Lemma}
\newtheorem{prop}{Proposition}
\newtheorem{corollary}{Corollary}
\theoremstyle{definition}
\newtheorem{definition}{Definition}
\newtheorem{example}{Example}

\theoremstyle{remark}
\newtheorem{remark}{Remark}

\numberwithin{equation}{section}

\everymath{\displaystyle}

\newcommand*\un{\operatornamewithlimits{\cup}}

\begin{document}
\title{ Completions of Leavitt path algebras}

%

\author{Adel Alahmadi}
\address{Department of Mathematics, Faculty of Science, King Abdulaziz University, P.O.Box 80203, Jeddah, 21589, Saudi Arabia}
\email{analahmadi@kau.edu.sa}
\author{Hamed Alsulami}
\address{Department of Mathematics, Faculty of Science, King Abdulaziz University, P.O.Box 80203, Jeddah, 21589, Saudi Arabia}
\email{hhaalsalmi@kau.edu.sa}

%


\maketitle

\begin{abstract}
 We introduce a class of topologies on the Leavitt path algebra $L(\Gamma)$ of a finite directed graph and decompose a graded completion $\widehat{L}(\Gamma)$ as
  a direct sum of minimal ideals.
\end{abstract}

\maketitle

\section{Definitions and Terminology}
Let $\Gamma=(V, E, s, r)$ be a finite directed graph, that consists of two sets $V$ and $E,$ called vertices and edges respectively and two maps $s,r: E\to V.$ The vertices $s(e)$ and $r(e)$
are referred to as the source and range of the edge $e$ respectively.

\bigskip

 A vertex $v$ such that $s^{-1}(v)=\emptyset$ is called a sink. A path $p=e_1\cdots e_n$ in a graph $\Gamma$ is a sequence of
edges $e_1,\ldots, e_n$ such that $r(e_i)=s(e_{i+1})$ for $i=1,\ldots, n-1.$ We will refer to $n$ as the length of the path $p,$ $l(p)=n.$ Vertices are viewed as paths of length $0.$
We say that the path $p$ starts at the source $s(p)=s(e_1)$ and ends at the range $r(p)=r(e_n).$  The set of all paths of the graph $\Gamma$ is denoted sa $Path(\Gamma).$
If $s(p)=r(p)$ then we say that the path $p$ is closed. If $p=e_1\cdots e_n$ is a closed path of length $\geq 1$ and the vertices $s(e_1),\cdots, s(e_n)$ are distinct then we call the path
$p$ a cycle. Denote $V(p)=\{s(e_1),\cdots, s(e_n)\},$ $E(p)=\{e_1,\cdots, e_n\}.$ An edge $e\in E$ is called an exit of a cycle $C$ if $s(e)\in V(C),$ but $e\notin E(C).$

\bigskip

If $X,Y$ are nonempty subsets of the set $V$ then we denote $E(X,Y)=\{e\in E \mid s(e)\in X, r(e)\in Y\},$ $Path(X,Y)=\{p\in Path(\Gamma)\mid s(p)\in X, r(p)\in Y\}.$
\bigskip
A vertex $w\in V$ is called a descendant of a vertex $v\in V$ if $Path(\{v\},\{w\})\neq\emptyset.$

\bigskip
A nonempty subset $W\subseteq V$ is said to be hereditary if for an arbitrary element $w\in W$ all descendants of $w$ lie in $W$ (see[A] ).

\bigskip
Let $F$ be a field. The Leavitt path algebra $L(\Gamma)$ is the $F$-algebra presented by the sets of generators $\{v \mid v\in V\},\, \{e, e^* \mid e\in E\}$ and the set of relations
(1)  $v_iv_j=\delta_{ij} v_i$ for all $v_i, v_j\in V;$ (2) $s(e)e=er(e)=e, \, r(e)e^*=e^*s(e)=e^*$ for all $e\in E;$ (3) $e^*f=\delta_{e,f} r(e)$ for all $e,f \in E;$ (4) $v=\sum\limits_{s(e)=v} ee^*$ for an arbitrary
vertex $v$ which is not a sink, ([AA, AMP, A]).

\bigskip
The mapping $*$ which sends $ v$ to $v$ for $v\in V,$ $e$  to $e^*$ and $e^*$ to $e$ for $e\in E,$ extends to an involution of the algebra $L(\Gamma).$

\section{ Topology on $L(\Gamma)$}
We call a mapping $\gamma: V\setminus\{ sinks\}\to E$ a \textit{specialization} if $s(\gamma(v))=v$ for an arbitrary vertex $v\in V\setminus\{ sinks\}.$ Edges lying in the image
$\gamma(V\setminus\{ sinks\}) $ are called special.  For a specialization $\gamma$ consider the set $B(\gamma)$ of the products $pq^*,$ where $p=e_1\cdots e_n,\, q=f_1\cdots f_m$ are paths in
$\Gamma ;$ $e_i, f_j \in E ; r(p)=r(q)$ and either $e_n\neq f_m$ or $e_n=f_m,$ but this edge is not special.
\bigskip

In [AAJZ1] we proved that $B(\gamma)$ is a basis of the algebra $L(\Gamma).$
\bigskip

We call a path $p=e_1\cdots e_n$ of length $n\geq 1$ \textit{special} if all edges $e_1,\ldots, e_n$ are special. For an arbitrary path $p=e_1\cdots e_n$ let $i$ be the minimal integer such that the path
$e_{i+1}\cdots e_n$ is special. If the edge $e_n$ is not special then $i=n.$ Let $sd(p)=n-i.$
\bigskip

The algebra $L(\Gamma)$ is $\mathbb{Z}-$graded: $\deg(V)=0, \deg(E)=1,$ and $\deg(E^*)=-1.$
\bigskip

Let $p,q\in Path(\Gamma).$ We say that the path $p$ is a \textit{beginning} of the path $q$ and the path $q$ is a \textit{continuation} of the path $p$ if there exists a path $q'\in Path(\Gamma)$ such that
$q=pq'.$

\begin{remark}
We will often use the following straightforward fact: if $p,q\in Path(\Gamma)$ then $p^*q\neq 0$ if and only if one of the paths $p,q$ is a continuation of the other one.

\end{remark}

For nonnegative real numbers $n, s, d$ consider the subspace $V_{n,s,d}$ of $L(\Gamma)$ $F$-spanned by all products $pq^*$ such that $ p,q \in Path(\Gamma),\, l(p)+l(q)\geq n, sd(p)+sd(q)\leq s,\, |\deg(pq^*)|=|l(p)-l(q)|\leq d.$

\begin{lemma}
$V_{n_1,s_1,d_1} . V_{n_2,s_2,d_2}\subseteq V_{\frac{1}{2}(n_1+n_2-d_1-d_2),s_1+s_2,d_1+d_2}.$
\end{lemma}

\begin{proof}
Let $p_i, q_i\in Path(\Gamma),$  $p_iq_i^*\in V_{n_i,s_i,d_i}, i=1,2.$ Then $l(p_i)+l(q_i)\geq n_i,  |l(p_i)-l(q_i)|\leq d_i,$ which implies $l(p_i), l(q_i)\geq \frac{1}{2}(n_i-d_i).$
If $p_1q^*_1p_2q^*_2\neq 0$ then in view of the Remark 1 there exists a path $p'_2\in Path(\Gamma)$ such that  $p_2=q_1p'_2$ or there exists a path $q'_1\in Path(\Gamma)$ such that  $q_1=p_2q'_1.$

We will consider only the first case $p_2=q_1p'_2.$ The second case is treated similarly. We have  $p_1q_1^*p_2q^*_2=p_1p'_2q^*_2.$ Clearly  $$l(p_1p'_2)+l(q_2)\geq l(p_1)+l(q_2)\geq \frac{1}{2}(n_1-d_1)+\frac{1}{2}(n_2-d_2).$$
Furthermore, $sd(p_1p'_2)+sd(q_2)\leq sd(p_1)+sd(p'_2)+sd(q_2)\leq s_1+s_2.$ \newline Finally, $\deg(p_1q_1^*p_2q^*_2)=\deg(p_1q^*_2)+\deg(p_2q^*_2),$ hence
$|\deg(p_1q_1^*p_2q^*_2)|=|\deg(p_1q^*_2)+\deg(p_2q^*_2)|\leq d_1+d_2.$
This completes the proof of the Lemma.

\end{proof}

Let $k\geq 1.$ Let $V_k=\sum\{ V_{n,s,d} \mid n\geq k(s+d+1)\}.$ Clearly, $V_{k_1}\subset V_{k_2}$ for $k_1\geq k_2.$

\begin{lemma}
Let $k\geq 3.$ Then $V_k V_k\subseteq V_{\frac{1}{2}(k-1)}.$

\end{lemma}

\begin{proof}
Suppose that $n_i\geq k(s_i+d_i+1), i=1,2.$ Then $$V_{n_1,s_1,d_1} .V_{n_2,s_2,d_2} \subseteq V_{\frac{1}{2}(n_1+n_2-d_1-d_2),s_1+s_2,d_1+d_2}.$$
We have $\frac{1}{2}(n_1+n_2-d_1-d_2)\geq \frac{1}{2}(n_1+n_2-d_1-d_2-s_1-s_2)>\frac{1}{2}(k-1)(s_1+s_2+d_1+d_2).$

\end{proof}

\begin{lemma}
For an arbitrary element $a\in L(\Gamma),$ an arbitrary $k\geq 1$ there exists $k'\geq 1$ such that $aV_{k'} +V_{k'}a\subseteq V_k.$

\end{lemma}

\begin{proof}

Without loss of generality we can assume that $a=pq^*,$ where $p,q$ are paths. Then $a\in V_{n_0,s_0,d_0},\,\, n_0=l(p)+l(q), \, s_0=sd(p)+sd(q),\, d_0=|\deg(a)|.$
\bigskip

Let $n\geq 0, s\geq 0, d\geq 0.$ By Lemma 1 we have  $V_{n_0,s_0,d_0} V_{n,s,d}\subseteq V_{\frac{1}{2}(n+n_0-d-d_0),s+s_0,d+d_0}.$ For the right hand side  to lie in $V_k$ it is sufficient to have
$\frac{1}{2}(n-(s+d)+n_0-d_0)\geq k(s+d+s_0+d_0+1)$ or, equivalently $\frac{1}{2}(n-(s+d))\geq k(s+d)+c,$ where $c=k(s_0+d_0+1)-\frac{1}{2}(n_0-d_0).$

If $n\geq k'(s+d+1),$ then $\frac{1}{2}(n-(s+d))\geq\frac{1}{2}(k'-1)(s+d)+\frac{1}{2}k'.$ Hence, for $k'\geq \max\{2k+1,2c\}$ the inclusion of the Lemma holds.

\end{proof}

\begin{lemma}
$\bigcap\limits_{k\geq 1} V_k=(0).$
\end{lemma}

\begin{proof}
Recall that the basis $B(\gamma)$ of the algebra $L(\Gamma)$ that corresponds to the specialization $\gamma$ consists of products $pq^*,$ where $p=e_1\cdots e_n, q=f_1\cdots f_m \in Path(\Gamma);$
$e_i, f_j\in E;$ $r(p)=r(q)$ and either $e_n\neq f_m$ or $e_n=f_m,$ but this edge is not special.
Let $V_{(n)}$  denote the $F$-algebra of all products $pq^*\in B(\gamma)$ such that $ l(p)+l(q)\geq n.$
\bigskip

Clearly, $\bigcap\limits_{n\geq 1} V_{(n)}=(0).$ It is easy to see that $V_{n,d,s}\subseteq V_{(n-s)}.$ Hence $V_k\subseteq V_{(k)}$ for
$k\in\mathbb{Z}, k\geq 1,$
which implies the assertion of the Lemma.
\end{proof}

The subspaces $\{V_k\}_{k\geq1}$ form a basis of neighborhoods of $0$ in $L(\Gamma)$ and define a topology. By Lemmas 2, 3 this topology  is
compatible with the algebra structure. Let $\overline{L(\Gamma) }$ be the completion of the topological  algebra $L(\Gamma).$ Let $\overline{L(\Gamma) }_i$ denote the completion
of the homogeneous component $L(\Gamma)_i$ of degree $i$ in the algebra $\overline{L(\Gamma) }.$ The main focus of this paper will be on the completion $\widehat{L}(\Gamma)=\sum\limits_{i\in \mathbb{Z}} \overline{L(\Gamma)_i}.$

\begin{example}
Let $\Gamma=$ \includegraphics[width=.05\textwidth]{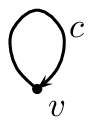} be a loop. The only edge $c$ is special. Hence
$ V_k=(0)$ for $k\geq1.$ The topology is discrete.
\end{example}

\begin{example}
Let $\Gamma=$ \includegraphics[width=.05\textwidth]{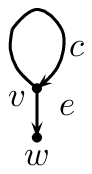}. The Leavitt path algebra $L(\Gamma)$ is the so called algebraic Toplitz algebra. It is isomorphic to the Jacobson algebra [J].
Let $I$ be the ideal of $L(\Gamma)$ generated by the vertex $w.$ Then $I$ is isomorphic to the algebra of finitary (having finitely many nonzero entries) infinite matrices $M_{\infty}(F)$ and
$(0)\to M_{\infty}(F)\to L(\Gamma)\to F[t^{-1},t]\to (0)$ is the nonsplit extension (see [SM, AAZ2]).  Let $e$ be the special edge. The graded completion
$\widehat{L}(\Gamma)$ is isomorphic to the algebra of infinite (not necessarily finitary) matrices having finitely many nonzero diagonals, $\widehat{L}(\Gamma)=\{ (a_{ij})_{i,j\geq 1} \mid a_{ij}\in F, \text{ there exists } d\geq 1 \text{ such that } a_{ij}=0 \text{ whenever } |i-j|> d\}.$
\end{example}

\begin{lemma}
Let $k\geq 2.$ Let $a$ be an element from $V_k$ and let $a=\sum\limits_{i} \alpha_i a_i$ be the presentation of $a$ as a linear combination of basic elements $a_i\in B(\gamma), \, \alpha_i\in F.$
Then all the elements $a_i$ lie in $V_{k-1}.$
\end{lemma}

\begin{proof}
Without loss of generality we will assume that $a=pq^*; \, p,q\in Path(\Gamma),\, r(p)=r(q),\, l(p)+l(q)\geq k(sd(p)+sd(q)+\deg(pq^*)+1).$ The presentation  $a=\sum\limits_{i} \alpha_i a_i$
is obtained by the Groebner-Shirshov algorithm ( see [BC, AAJZ1]). All the basic elements $a_i$ are of the types $a_i =p_iq^*_i,$ where $p_i,q_i\in Path(\Gamma),\, l(p_i)=l(p)-r,\, l(q_i)=l(q)-r,$
$\deg(p_iq_i^*)=\deg(pq^*).$ Now, $l(p_i)+l(q_i)=l(p)+l(q)-2r\geq k(sd(p_i)+sd(q_i)+2r+\deg(p_iq^*_i)+1)-2r\geq (k-1)(sd(p_i)+sd(q_i)+\deg(p_iq^*_i)+1),$ which finishes the proof of the Lemma.
\end{proof}

\begin{lemma}
$B(\gamma)$ is a topological basis of the algebra $\widehat{L}(\Gamma),$ i.e. an arbitrary element of $\widehat{L}(\Gamma)$ can be uniquely represented as a converging series
$\sum\limits_{b\in B(\gamma)} \alpha_b b,\, \alpha_b\in F.$
\end{lemma}

\begin{proof}
An arbitrary element of $\widehat{L}(\Gamma)$ can be represented as a converging sum $\sum\limits_{i\in\Omega} a_i,$ where $\{a_i\in L(\Gamma), i\in\Omega\}$ is a Cauchy set.
In other words for an arbitrary $k\geq 1$ the set $\{ i\in \Omega \mid a_i\notin V_k\}$ is finite. Let $a_i=\sum\limits_{j}\alpha_{ij} a_{ij},\, 1\leq j\leq t_i,$ be the decomposition of $a_i$ as a linear combination
of distinct basic elements from $B(\gamma),\, \alpha_{ij}\neq 0,\, a_{ij}\in B(\gamma).$ From Lemma 5 it follows that $\{a_{ij}, \, i\in \Omega,\, 1\leq j\leq t_i\}$ is also a Cauchy set.
Hence $\sum a_i=\sum \alpha_{ij} a_{ij}.$

Let $\sum \alpha_b b=0,\, \alpha_b\in F,\, b\in B(\gamma).$ Suppose that $\alpha_{b_0}\neq 0\neq$ and $b_0\notin V_k.$ Then by Lemma 5 any finite subsum of $\sum \alpha_b b,$ containing $\alpha_{b_0} b_0,$ does not belong to $V_{k+1},$ a contradiction. This completes the proof of Lemma.
\end{proof}

We will need another general statement about sums in $\widehat{L}(\Gamma).$ Consider a nonzero converging sum $a=\sum\limits_{i\in \Omega} a_i\in \widehat{L}(\Gamma),\, a_i\in L(\Gamma).$
We say that the sum is \textit{reduced} if for any arbitrary nonempty subset $\Omega'\subset\Omega$ we have $\sum\limits_{i\in\Omega'} a_i\neq 0.$

\begin{lemma}
For an arbitrary nonzero converging sum $a=\sum\limits_{i\in \Omega} a_i,\, a_i\in L(\Gamma),$ there exists a nonempty subset $\Omega'\subset\Omega$ such that $a=\sum\limits_{i\in \Omega'} a_i$ and this sum is reduced.
\end{lemma}

\begin{proof}
Let $\Omega_1\subset\Omega_2\subset \cdots$ be an ascending chain of subsets of $\Omega$ such that $\sum\limits_{i\in\Omega_k}a_i=0$ for each $k.$ Denote $\widetilde{\Omega }=\cup_{k} \Omega_k.$

We claim that $\sum\limits_{i\in \Omega} a_i=0.$ Indeed, since the sum $\sum\limits_{i\in \Omega} a_i$ is convergent it follows that for an arbitrary $t\geq 1$ the set $\{i\in\Omega \mid a_i\notin V_t\}$ is finite.
Hence there exists $k\geq 1$ such that $a_i\in V_t$ for any $i\in \widetilde{\Omega}\setminus\Omega_k.$

Now, $\sum\limits_{i\in \widetilde{\Omega}}a_i=\sum\limits_{i\in\Omega_k}a_i+\sum\limits_{i\in\widetilde{\Omega}\setminus\Omega_k}a_i\in \overline{V_t}.$ This implies that $ \sum\limits_{i\in \widetilde{\Omega}}a_i\in \cap_{t\geq 1} \overline{V_t}=(0).$ By Zorn's Lemma there exists a maximal subset $\Omega_{max}\subset\Omega$ such that $\sum\limits_{i\in\Omega_{max}} a_i=0.$
Let $\Omega'=\Omega\setminus\Omega_{max}.$ Then $a=\sum\limits_{i\in\Omega'} a_i$ and this sum is reduced.
\end{proof}

\section{ Central Idempotents in $\widehat{L}(\Gamma)$}

\begin{lemma}
Let $p=e_1\cdots e_n$ be a special path and $r(e_n)\notin \{s(e_1),\cdots,s(e_n)\}.$ Then $n\leq |V|.$
\end{lemma}

\begin{proof}
If $n>|V|$ then some vertex on $p$ appears at least twice and this vertex is not $r(e_n).$ Hence, a subpath $p_1$ of $p$ is a cycle.
Since $r(e_n)$ does not lie in $V(p_1)$ it follows that some exit from the cycle $p_1$ is  special. But this is impossible since for every non-sink
vertex $v$ only one edge from $s^{-1}(v)$ is special.
\end{proof}

\begin{definition}
Let $W\subset V$ be a nonempty subset. We say that a path $p=e_1\cdots e_n,\, e_i\in E, $ is an \textit{arrival path} in $W$ if $r(p)\in W,$ and $\{s(e_1),\cdots, s(e_n)\}\nsubseteq W.$
In other words, $r(p)$ is the first vertex on $p$ that lies in $W.$ In particular, every vertex $w\in W,$ viewed as a path of zero length,  is an arrival path in $W.$ Let $Arr(W)$ be the set of all arrival paths in $W.$
\end{definition}

\begin{lemma}
The set $\{ pp^*\mid p\in Arr(W)\}$ is a Cauchy set.
\end{lemma}

\begin{proof}
We need to check that for an arbitrary $ k\geq 1$ the set $\{pp^* \mid p\in Arr(W)\}\setminus V_k$ is finite. If $p$ is an arrival path in $W,$ then by Lemma 8 $sd(pp^*)\leq 2|V|, d(pp^*)=0.$
Hence $\{pp^* \mid p\in Arr(W)\}\setminus V_k\subseteq \{pp^* \mid l(p)< k( |V|+\frac{1}{2})\}.$ Clearly, it is a finite set, which completes the proof.
\end{proof}

Denote $e(W)=\sum\limits_{p\in Arr(W)} pp^* \in\widehat{ L}(\Gamma).$

\begin{lemma}
If $W$ is a hereditary set, then $e(W)$ is a central idempotent in  $\widehat{ L}(\Gamma).$
\end{lemma}

\begin{proof}
If $a,b$ are distinct elements from $\{ pp^*\mid p\in Arr(W)\}$ then $ab=ba$ by Remark 1.  Hence, $e(W)$ is a sum (possibly infinite) of pairwise orthogonal idempotents.
Hence $e(W)$ is an idempotent. Since $L(\Gamma)$ is dense in $\widehat{ L}(\Gamma)$ it is sufficient to show that $e(W)$ commutes with all vertices and all
edges of $\Gamma.$ For a vertex $v\in V$ let $Arr(v,W)=\{ p\in Arr(W)\mid s(p)=v\}.$  If $w\in W$ then $Arr(w,W)=\{w\}.$ It is easy to see that $v.e(W)=e(W).v=\sum\limits_{\scriptscriptstyle p\in Arr(v,W)} pp^*.$
Let $e\in E.$  We will consider $3$ cases:

\textbf{Case 1.} $r(e)\notin W.$ Then $e\sum\limits_{\scriptscriptstyle p\in Arr(W)}pp^*=e\sum\limits_{p\in Arr(r(e),W)}pp^*;$  $\sum\limits_{\scriptscriptstyle p\in Arr(W)}pp^*e=\sum\{pp^*e\mid p\in Arr(W), \text{ the first edge of } p\text{ is } e\}.$
It is easy to see that these two sums are equal.

\textbf{Case 2.} $r(e)\in W,\, s(e)\notin W.$ Then $e\in Arr(W).$ We have $e\sum\limits_{\scriptscriptstyle p\in Arr(W)} pp^*=er(e)r(e)^*=e;$ $\sum\limits_{\scriptscriptstyle p\in Arr(W)} pp^* e=ee^*e=e.$

\textbf{Case 3.} $r(e)\in W,\, s(e)\in W.$ In this case we again have \,\, $e\sum\limits_{\scriptscriptstyle p\in Arr(W)} pp^*=e;$ $\sum\limits_{\scriptscriptstyle p\in Arr(W)} pp^* e=s(e)s(e)^*e=e.$

\end{proof}

\section{Frames}

Let $ W$ a nonempty subset of $V$. We will define a graph $\Gamma^W=(V',E')$  as follows:

$V'=(V\setminus W)\cup \{w\},$ where $w$ is a new vertex, not belonging to $V;$ for two vertices $v_1,v_2\in V\setminus W,$ the set of edges $E'(\{v_1\},\{v_2\})$ is identified with $E(\{v_1\},\{v_2\});$
the set of edges $E'(\{v_1\},\{w\})$ is identified with $E(\{v_1\},W).$ For an edge $e\in E(\{v_1\},W)$ and its image $e'$ in $E'(\{v_1\},\{w\})$ we will say that $e'$ is the edge $e$ redirected to $w.$

\begin{remark}
Since all edges $E(W, V\setminus W)$ are ignored, the vertex $w$ is a sink in $\Gamma^W.$
\end{remark}

\begin{lemma}
Let $\Gamma=(V,E)$ be a finite graph with a sink $v$ which is a descendant of every vertex in $V.$
Then there exists a specialization $\gamma: V\setminus\{v\}\to E,$ such that the set of all special paths in $\Gamma$ is finite.
\end{lemma}

\begin{proof}
Let $v_1,\cdots, v_k\in V$ be vertices such that $E(\{v_i\},\{v\})\neq\emptyset.$ In each set $E(\{v_i\},\{v\})$ choose one edge and declare it special.
All other edges coming out of $ v_1,\cdots, v_k$ are not special. Consider the graph $ \Gamma'=\Gamma^{\{v_1,\cdots, v_k, v\}}=(V', E'), V'=(V\setminus \{v_1,\cdots, v_k\})\cup\{w\}.$
Since $v$ is a descendant of an arbitrary vertex in $V$ it follows that $v$ is the only sink in $\Gamma.$
Similarly, $w$ is a descendant of an arbitrary vertex in $V',$ hence, $w$ is the only sink in $\Gamma'.$
Since $|V'|<|V|$ by the induction assumption there exists a specialization  $\gamma': V'\setminus\{w\}\to E'$ such that  the set of  special paths in $\Gamma'$ is finite.
Now we are ready to construct the specialization $\gamma: V\setminus\{v\}\to E.$
Choose a vertex $u\in V\setminus\{ v_1,\cdots, v_k, v\}$ and let $\gamma'(u)=e'\in E'.$
If $r(e')\neq w$ then $e'\in E(V\setminus\{ v_1,\cdots, v_k, v\},V\setminus\{ v_1,\cdots, v_k, v\})$ and we define $\gamma(u)=e'.$
Now let $r(e')=w.$ It means that there was an edge $u\xrightarrow{e}v_i,\, 1\leq i \leq k,$ that was redirected to $u\xrightarrow{e'}w.$
We let $\gamma(u)=e.$ If $u\in\{v_1,\cdots,v_k\}$ then at the  beginning of the proof we chose a special edge $u\to v.$
We claim that with the specialization $\gamma$ defined above there are finitely many special paths in $\Gamma.$

If  not, then $\Gamma$ contains a special cycle. This cycle can not involve any of the
vertices $v_1,\cdots,v_k,$ since special edges from $v_1,\cdots,v_k$ lead to $v,$ a sink. Hence this cycle lies in $\Gamma',$ which contradicts the induction assumption.
This proves the Lemma.
\end{proof}

Let $W$ be a minimal hereditary subset of $V.$ Then for any two vertices  $w_1,w_2\in W$  the vertex  $w_2$ is a descendant of  $w_1.$
Indeed, the set of all descendants of  $w_1$ is a hereditary subset of $V.$  In view of minimality of $W$ it contains $W.$ It implies that for any two minimal hereditary subset
$W_1, W_2$ either $W_1=W_2$ or $W_1\cap W_2=\emptyset.$
\bigskip

Let $W_1,\cdots, W_k$ be all distinct minimal hereditary subsets of $V.$ We will call the subsets $W_1,\cdots, W_k$ the \textit{frame} of $\Gamma.$

\begin{lemma}
Every vertex of $\Gamma$ has a descendant in $\un\limits_{i\geq 1}^k W_i.$
\end{lemma}

\begin{proof}
The set of all vertices that do not have a descendant in $\un\limits_{i\geq 1}^k W_i$ is hereditary. If nonempty, then it contains one of the subsets $W_1,\cdots, W_k,$
a contradiction. This proves the Lemma.
\end{proof}

\begin{lemma}
There exists a specialization $\gamma:V\setminus\{sinks\}\to E$ such that the set of all special paths $p=e_1\cdots e_n$ with $s(e_1),\cdots, s(e_n)\notin \un\limits_{i\geq 1}^k W_i$ is finite.
\end{lemma}

\begin{proof}
Consider the graph $\Gamma'=\Gamma^{W_1\cup\cdots\cup W_k}=(V',E'),\, V'=(V\setminus(\un\limits_{i\geq 1}^k W_i))\cup\{w\}.$ This graph contains a sink $w,$ which is a descendant of all
vertices in $V'$ by Lemma 12. By Lemma 11 there exist a specialization $\gamma':V'\setminus\{w\}\to E'$ such that the set of all special paths in $\Gamma'$ is finite.

Let's define a specialization $\gamma:V\setminus\{sinks\}\to E.$ For non-sinks from $\un\limits_{i\geq 1}^k W_i$ define $\gamma$ arbitrarily. Choose a vertex $u\in V\setminus(\un\limits_{i\geq 1}^k W_i).$
Clearly, $u$ is not a sink in $\Gamma'.$ If $r(\gamma'(u))\neq w $ then we let $\gamma(u)=\gamma'(u).$ If $r(\gamma'(u))=w$ then $\gamma'(u)$ has been redirected from some edge $e\in E,\, s(e)=u,\, r(e)\in\un\limits_{i\geq 1}^k W_i.$ Let $\gamma(u)=e.$ If $p=e_1\cdots e_n$ is a special path in $\Gamma$ such that $s(e_1),\cdots, s(e_n)\in V\setminus( \un\limits_{i\geq 1}^k W_i)$ then $p$ can be viewed as a special path in $\Gamma'.$ Since there are finitely many such paths, this completes the proof of the Lemma.
\end{proof}

From now on we will talk only about specializations that satisfy the condition of Lemma 13.

\begin{lemma}
Let $W'\subset W''\subset V$ be hereditary subsets such that every vertex from $W''$ has a descendant in $W'.$ Then $e(W')=e(W'').$
\end{lemma}

\begin{proof}
Let $W_1,\cdots,W_k$ be the frame of the graph $\Gamma$. Let $\gamma:V\setminus\{sinks\}\rightarrow E$ be a specialization that satisfies the condition of Lemma 13.

Since every vertex in $W''$ has a descendant in $W'$ it follows  that for an arbitrary minimal hereditary  subset $W_i$ either $W_i \cap W''=\phi$ or $W_i \subseteq W'.$ Choose a vertex $v \in W'' \setminus W'.$ As above we denote $Arr(v,W')=\{p \in Arr(W')\mid s(p)=v.\}$ We claim that $v=\sum\limits_{\scriptscriptstyle p \in Arr(v,W')}pp^*.$ To prove this equality we will define a sequence of finite sets of paths $P_0, P_1,\cdots$. Let $P_0=\{v\}$. If $P_n$ has been constructed then $P_{n+1}$ is defined in the following way. Let $p \in P_n$. If $r(p) \in W'$ then $p \in P_{n+1}$. Let $r(p)\in W''\setminus W'.$ Then $r(p)$ is not a sink (the set $W'' \setminus W'$ does not contain sinks). Let $e_1,\ldots,e_q$ be all edges with the source at $r(p)$. Then $pe_1,...,pe_q \in P_{n+1}$. Thus $P_{n+1}=\{p \in P_n,r(p) \in W'\}\dot{\cup} \{pe \mid p\in P_n,r(p) \in W''\setminus W', s(e)=r(p)\}. $
 For an arbitrary $n\geq 0$ we have $v=\sum\limits_{\scriptscriptstyle p \in P_n}pp^*.$ If $p \in P_n$ and $r(p) \in W'$ then $p$ is an arrival path in $W'$. By Lemma 10 all $sd(p), p \in \un\limits_{n\geq 0}P_n$, are uniformly bounded from above. Hence
 $\sum\limits_{\scriptstyle p\in P_n\setminus Arr(W')}pp^*\to  0.$ It follows that $$v=\lim \limits_{n \to \infty} \sum\limits_{\scriptstyle p\in P_n\cap Arr(W')}pp^*=\sum\limits_{\scriptstyle p\in Arr(v,W')}pp^*.$$

 Now, $$e(W')=\sum\limits_{\scriptstyle p\in Arr(W'')}p\left(\sum\limits_{\scriptstyle p_1\in  Arr(r(p),W')}p_1p^*_1\right)p^*=\sum\limits_{\scriptstyle p\in  Arr(W'')}pp^*=e(W'').$$

\end{proof}

 Let $W$ be a nonempty hereditary subset of $V$. Let $W^\bot \subset V$ consist of those vertices which do not have descendants in $W$. Clearly, $W^\bot$ is a hereditary subset of $V$.

\begin{lemma}
The idempotents $e(W),e(W^\bot)$ are orthogonal and $e(W)+e(W^\bot)=1$ ( if $W^\bot=\phi$ then we let $e(W^\bot)=0.$)

\end{lemma}

\begin{proof}
If $p,q$ are arrival paths to $W,W^\bot$ respectively then none of them is a continuation of the other one. Hence $p^*q=q^*p=0$. It implies that $e(W)e(W^\bot)=e(W^\bot)e(W)=0$ .\\
An arbitrary vertex from $V$ has a descendant  in $W \cup W^\bot$. Indeed, if $v \in V$ and $v$ does not have descendants in $W$ then $ v \in W^\bot$. By Lemma14, $e(W)+e(W^\bot)=e(W \cup W^\bot)=e(V)=1.$ This finishes the proof of the Lemma.
\end{proof}

\begin{corollary}
$e(W)=e((W^\bot)^\bot).$

\end{corollary}

\begin{lemma}
$(W^\bot)^\bot$ is the largest hereditary subset of $V$ such that every vertex of it has a descendant in $W$.
\end{lemma}

\begin{proof}
 Since $(W^\bot)^{\bot}\cap W^\bot=\phi$ we conclude that every vertex from $(W^\bot)^{\bot}$ has a descendent in  $W$. Now let $U \subseteq V$ be a nonempty hereditary subset such that every vertex from $U$ has a descendant in $W$. In order to prove $V \subseteq (W^\bot)^{\bot}$ we need to show that no vertex $u \in U$ can have a descendant in  $W^\bot$. Let $v$ be a descendant of the vertex $u$ that lies in $W^\bot$ . Since $U$ is hereditary it follows that $v \in U$. Hence, $v$ has a descendant in $W$. It contradicts the inclusion $v \in W^\bot$ and completes the proof.
\end{proof}

The closed ideal of the algebra $ \widehat{L}(\Gamma)$ generated by the hereditary subset $W \subset V$ consists of (possibly infinite) converging sums $\sum \alpha_{pq}pq^*$, where $\alpha_{pq}\in F; p,q \in Path(\Gamma), r(p)=r(q) \in W.$ We will denote this ideal as $I(W).$

\begin{lemma}
$I(W)=e(W)\widehat{L}(\Gamma).$
\end{lemma}

\begin{proof}
For an arbitrary vertex $w \in W$ we have $w=w e(W).$ Hence $W \subset e(W)\widehat{L}(\Gamma)$ and $I(W) \subseteq e(W)\widehat{L}(\Gamma).$ The inclusion $e(W) \in I(W)$ follows from the fact that every arrival path in $W$ ends with a vertex from $W$. This finishes the proof of the Lemma.
\end{proof}

Lemma 17 implies that the ideal $I(W)$ is a direct summed of the algebra $\widehat{L}(\Gamma)=I(W)\oplus I(W^\bot)$ and that $\widehat{L}(\Gamma)=I(W_1)\oplus \cdots \oplus I(W_k).$ Now our aim is to decompose $\widehat{L}(\Gamma)$ as a direct sum of minimal ideals.

\section{Completions of simple leavitt path algebras and the ideals $I(W_i).$}

Recall that $\gamma:V\rightarrow E$ is a fixed specialization of the algebra $\Gamma.$ If $W$ is a hereditary subset of $V$ then $\gamma(W)\subseteq E(W,W).$\\
For a vertex $v \in V$ define a special path $g_v (n)$ inductively. Let $g_v (0)=v.$ If $r (g_v (n))$ is not a sink then $g_v (n+1)=g_v (n)\gamma (r(g_v (n)))$. If $r(g_v (u))$ is a sink then $g_v (n+1)=g_v (n)$.

\begin{lemma}
$\{g_v (n)g_v (n)^*,n\geq 0\}$ is a Cauchy set.
\end{lemma}

\begin{proof}
For a vertex $v \in V$ let $\mathcal{E}(v)$ denote the set of all non special edges $e$ with $s(e)=v.$ Then $$g_v(n+1)g_v(n+1)^*=g_v(n)g_v(n)^*-\sum\limits_{\scriptstyle e \in \mathcal{E}(r(g_v(n)))}g_v(n)e e^*g_v(n)^*,$$
which implies $g_v(n+1)g_v(n+1)^*-g_v(n)g_v(n)^* \in V_{2(n+1)}.$
\end{proof}

Let $e_v=\lim \limits_{n \to \infty}g_v(n)g_v(n)^*$.

\begin{lemma}
(1) $e_v$ is a non zero idempotent , (2) if $v\neq w$ then $e_v, e_w$ are orthogonal , (3) if $e$ is a non special edge from $E(v,V)$ then $e_v e=0;$ if $e=\gamma(v)$ then $e_v e=e e_w,$ where $w=r(e)$.

\end{lemma}

\begin{proof}
For an arbitrary $n \geq 1$ the element $g_v(n)g_v(n)^*$ is an idempotent. Hence the limit $e_v$ is an idempotent as well. Let us show that $e_v\neq0.$ Indeed, denote $v_n=r(g_v(n)),v_0=v.$ Then $$g_v(n)g_v(n)^*=v-\sum\limits_{e \in\mathcal{E}(v_i)}g_v(i)ee^*g_v(i)^*.$$ Hence, $e_v=v-\sum\limits_{k\geq 1}a_k,$ where $a_k=g_v(k)(\sum\limits_{e \in \mathcal{E}(v_k)}ee^*)g_v(k)^*\to 0$ as $k\to \infty.$ Since $B(\gamma)$ is a topological basis of $\widehat{L}(\Gamma)$ by Lemma 6, it implies that $e_v\neq 0.$\\
For distinct vertices $v,w \in V$ we have $e_v e_w=0$ since $e_v \in v \widehat{L}(\Gamma)v.$\\
If $e \in E(v,V)$  and $e$ is not special then $e_v e=0$ since $g_v(n)^*e=0$ for $n \geq 1.$ If $e=\gamma(v)$ then $g_v(n)^*e=g_w(n-1)^*,$ where $w=r(e).$ Hence, $g_v(n) g_v(n)^*e=eg_w(n-1)g_w(n-1)^*.$ This finishes the proof of the Lemma.
\end{proof}

Consider the graph $(V,\gamma(V))$ with the set of vertices $V$ and the set of edges $\gamma(V).$ Let $\tilde{\gamma}(V)$ be the set $\gamma(V)$ with all edges having lost their directions, $(V,\tilde{\gamma}(V))$ is the corresponding not directed graph.\\

We say that a vertex $w$ is a special descendant of a vertex $v$ if there exists a special path $p$ in $\Gamma$ such that $s(p)=v, r(p)=w.$

\begin{lemma}
Vertices $v,w \in V$ are connected in $(V,\tilde{\gamma}(V))$ if and only if they have a common special descendant in $\Gamma.$
\end{lemma}

\begin{proof}
Let $p=\tilde{e}_1\cdots \tilde{e}_n$ be a geodesic path in $(V,\tilde{\gamma}(V))$ connecting $v$ and $w;\,  e_1,\cdots ,e_n \in \gamma(V).$ The (undirected) edge $\tilde{e}_1$ connects $v_1=v$ with a vertex $v_2,$ the edge $\tilde{e}_2$ connects $v_2$ with $v_3,$ and so on. All the vertices $v_1=v, v_2,\cdots,v_{n+1}=w$ are distinct. If $v_1\rightarrow v_2,$ then $v_2$ and $w$ have a common special descendant in $\Gamma$ as the distance between them in $(V,\tilde{\gamma}(V))$ is $n-1.$ Hence $v_1,w$ have a common special descendant.\\
Let $v_1\leftarrow v_2.$ Since there is a unique special edge in $\Gamma$ with the source $v_2$ it follows that $v_2\leftarrow v_3$ and similarly $v_3\leftarrow v_3\leftarrow\cdots v_n\leftarrow w.$
Now $v$ is a special descendant  of $w$ which finishes the proof of the Lemma.
\end{proof}

\begin{lemma}
(1) If vertices $v,w \in V$ are connected in $(V,\tilde{\gamma}(V))$ then $e_v,e_w$ generate the same closed ideal in $\widehat{L}(\Gamma)$; \\
(2). If $v,w \in V$ are not connected in $(V,\tilde{\gamma}(V))$ then $e_v\widehat{L}(\Gamma) e_w=(0).$
\end{lemma}

\begin{proof}
Let vertices $v,w \in V$ be connected in $(V,\tilde{\gamma}(V))$. Without loss of generality we can assume that there is a special edge $e \in \gamma(V)$ such that $v\rightarrow w ~or~ v\leftarrow w.$ In the first case $e_w=e^*e_v e$ by Lemma 19(3). In the second case $e_w=e e_v e^*.$ In both cases $e_w$ lies in the ideal generated $e_v$, which proves the claim (1).\\
Now let $v,w$ lie in different connected components of $((V,\tilde{\gamma}(V))$. Since $L(\Gamma)$ is dense in $\widehat{L}(\Gamma)$ it is sufficient to prove that $e_v L(\Gamma)e_w=(0)$. Let $p,q \in Path(\Gamma), r(p)=r(q), s(p)=v, s(q)=w.$ We need to show that $e_v pq^*e_w=0.$ If $p$ is not a special path then $e_v p=0$ by Lemma 19 (3) and similarly $q^*e_w=0$ if the path $q$ is not special. If both paths $p,q$ are special then Lemma 19(3) implies $e_v p= p e_{r(p)}, q^*e_w=e_{r(q)} q^*, r(p)\neq r(q)$. Hence $v,w$ do not have a common special descendant. Hence $e_v pq^*e_w=p e_{r(p)} e_{r(q)}q^*=0,$ which finishes the proof.
\end{proof}

\begin{lemma}
Let $I$ be a non zero closed graded ideal of $\widehat{L}(\Gamma).$ Then $I_0=I \cap\widehat{L}(\Gamma)_0\neq (0) $.

\end{lemma}
\begin{proof}
Choose a nonzero homogenous element $a \in I.$  Without loss of generality we can assume that there exist vertices $v,w \in V$ such that $a=v a w.$ If $deg(a)=0$ then we are done. Suppose that $\deg(a)=d\geq 1.$\\
\bigskip
The vertex $v$ can be represented as $v=\sum\limits_{i} p_i p_i^*,$ where $p_i \in Path(\Gamma)$ and for an arbitrary $i$ either $l(p_i)=d$ or $l(p_i)<d$ and $r(p_i)$ is a sink.
\bigskip
Let $a=\sum\alpha_{p,q} pq^*; \alpha_{p,q} \in F; p,q \in Path (\Gamma); r(p)=r(q); deg(pq^*)=d$ for every $p,q$. From $\deg(pq^*)=d\geq 1$ it follows that $l(p)=d+l(q)\geq d$ for each summed $d.$ \newline Suppose that $l(p_i)<d, r(p_i)$ is a sink and nevertheless $p_ip_i^* pq^*\neq 0.$ The path $p$ can not be a continuation of the path $p_i$ since $l(p_i)<l(p)$ and $r(p_i)$ is a sink. The path $p_i$ can not be continuation of path $p$ since $l(p)\geq d > l(p_i)$, a contradiction.

Hence , for all  $p_i p_i^*$ such that $l(p_i)<d,$ we have $p_i p_i^*a=0.$ This implies a $a \in \widehat{L}(\Gamma_d)\widehat{L}(\Gamma_{-d})a \subseteq \widehat{L}(\Gamma_d)I_0.$ The case $\deg(a)\leq -1$ is treated similarly.
\end{proof}

\begin{lemma}
 Let $W$ be a nonempty hereditary subset of $V$ and let $J$ be a nonzero closed graded of $\hat{L}(\Gamma)$ such that $J \subseteq I(W).$
 Then there exists a vertex $w\in W$ such that $e_w\in J.$
\end{lemma}

\begin{proof}
By Lemma 22 the space $J_0$ contains a nonzero element $ a=\sum\alpha_{p,q} pq^*, l(p)=l(q), r(p)=r(q) \in w.$ By Lemma 7 we can assume that the sum is reduced.
Denote $\mathcal{P}=\{(p,q)\in Path (\Gamma)\times Path(\Gamma)\mid \alpha_{p,q}\neq 0\}.$ Choose $(p_0,q_0)\in \mathcal{P}$ with minimal length $l(p_0).$ Let $r(p_0)=v \in W.$\\
Let $\mathcal{P}(p_0,q_0)=\{(p,q) \in \mathcal{P}\mid p\text{ and } q \text{ are proper continuations of paths } p_0,q_0 \text{ respectively } \},$ $\mathcal{P'}(p_0,q_0)=\{(p,q) \in Path(\Gamma) \times Path(\Gamma)\mid  (p_0 p, q_0 q)\in \mathcal{P} (p_0, q_0)\}.$ \newline Then $a'=p_0^*a q_0=\alpha_{p_0,q_0}v+\sum\limits_{\scriptstyle (p,q) \in \mathcal{P'}(p_0,q_0)}\alpha_{p,q}pq^*$ and $p_0 a' q_0^*=\alpha_{p_0,q_0}p_0 q_0^*+\sum\limits_{(p,q) \in \mathcal{P}(p_0,q_0)}\alpha_{p,q}pq^*\neq 0,$ since the sum is reduced. Hence, $a'\neq 0.$\\

\bigskip

Remark, that $a'=v a' v.$\\

\bigskip

Rewriting each summed $p q^*, (p,q) \in \mathcal{P'} (p_0,q_0),$ as a linear combination of basic elements from $B(\gamma)$ and using Lemmas 5, 6 we get $a'=\sum\beta_{p,q}pq^*,$ where $\beta_{p,q}\in F,\,\, l(p)=l(q), s(p)=s(q)=v, pq^* \in B(\gamma)$ for each summed and the sum is reduced. Remark that since the subset $W$ is hereditary it follows that $r(p)=r(q) \in W$ for each summed. As we did  before choose a summand $p_0'q_0'^*$ with minimal $l(p_0').$ If $p=p_0' p', q=q_0'q'$ and $pq^* \in B(\gamma)$ then $p'q'^* \in B(\gamma)$ as well. Now, $b=\frac{1}{\beta_{p_0' p'}}p_0'^* a' q_0'=w+\sum\mu_{p,q} pq^*$ is a nonzero element from $J_0, s(p)=s(q)=w, l(p)=l(q)\geq 1, pq^* \in B(\gamma)$ for each summand.

Since the sum $\sum \mu_{p,q}pq^{*}$ is convergent it follows that the set
$\{(p,q)\in Path (\gamma)\times Path (\gamma)\mid \mu_{p,q} \neq 0, pq^{*}\notin V_2\}$ is finite.\\

If $pq^{*} \in V_2$ then $2l(p)\geq 2(sd(p)+sd(q)+1),$ which implies that both paths $p,q$ are not special.\\

For an arbitrary basic element $t \in B(\gamma),$ of degree $0$ which is not a vertex, we have $g_w(n)^* t g_w(n)=0$ for a sufficiently large $n$. Indeed, let $t=pq^{*}, l(p)=l(q)\geq 1, n=l(p).$ If $g_w(n)^* t g_w(n)\neq 0$ then either $p=q=g_w(n)$ or $l(g_w(n))=r < n$ and $r(g_w(n))$ is a sink.\\
The first case is impossible  since $g_w(n) g_w(n)^* \notin B(\gamma).$ If $l(g_w(n))=r, 1\leq r <n,$ then $g_w(n)^*p=q^*g_w(n)$ since the path $g_w(n)$ ends with a sink and therefore can not be a beginning of paths $p,q.$ Finally, if $w$ is a sink then it can not be the source of paths $p,q.$ Hence, for a sufficiently large $n$ we have $g_w(n)^* b g_w(n)=r(g_w(n))+\sum \mu_{p,q} g_w(n)^* pq^* g_w(u).$ This expression is not equal to $0$ by Lemma 6. In each summand  $g_w(n)^* pq^* g_w(n)$ the special edges in $p,q$ won't cancel. Denote $u=r(g_w(n)).$ We have $0 \neq c=u+\sum \nu_{p,q} pq^* \in J_0;\,\, l(p)=l(q)\geq 1, sd(p)=sd(q)=n, r(p)=r(q),$ both $p$ and $q$ contain non special edges, $pq^* \in B(\gamma),$ for each summand.\\

Now, as we did above, consider $g_u(m)^* c g_u(m)=r(g_u(n))+\sum \nu_{p,q} g_u(m)^* pq^* g_u(m)$ and $g_u(m) g_u(m)^* c g_u(m) g_u(m)^*=g_u(m) g_u(m)^* + \sum \nu_{p,q} pq^*,$ where both $p,q$ in each summand are continuations of $g_u(m).$ If $r(g_u(m))$ is a sink then $g_u(m)^* c g_u(m)=r(g_u(m))=e_{r(g_u(m))}\in J.$ If for any $m \geq 1,\,\, r(g_u(m))$ is not a sink then the sequence $\sum \nu_{p,q} pq^*,$ where $p,q$ are continuations of $g_u(m),$ converges to $0$ as $m\rightarrow \infty.$ This implies $e_u=\lim \limits_{m \to \infty}g_u(m)g_u(m)^* \in J,$ and completes the proof of Lemma.

\end{proof}

\begin{corollary}
 The algebra $\widehat{L}(\Gamma)$ does not have non-zero closed graded nilpotent  ideals.
 \end{corollary}

 \begin{lemma}
 Let $W$ be a minimal hereditary subset of $V$. Then the ideal $I(W)$ is generated (as an ideal) by all idempotents $e_w, w \in W.$
 \end{lemma}

 \begin{proof}
 If $W$ consists of one sink $w$ then $e_w=w.$ Suppose therefore that the subset $W$ does not contain sinks. Let $J$ be the ideal of $\widehat{L}(\Gamma)$ generated by all idempotents $e_w, w \in W, J \subseteq I(W).$\\
\bigskip
For a vertex $w \in W$ we have $e_w=w-\sum \{g_w(k)e e^* g_w(k)^* \mid k \geq 0, e \in \mathcal{E}(r(g_w(k)))\}.$ For arbitrary vertices $w, v \in W,$ consider the operator $A_{w,v}:v \widehat{L}(r)v \rightarrow w \widehat{L} (r)w$,

$A_{w,v}(a)=\sum\{g_w(k)e a e^* g_w (k)^*\mid k\geq 0, e \in \mathcal{E}(r(g_w(k))), r(e)=v\}$. If the vertex $v$ does not appear as range of some path  $g_w(k) e, e \in \mathcal{E}(r(g_w(k))),$ then $A_{w,v}=0.$

\bigskip
Let $W=\{w_1,\cdots,w_r\}.$ Consider the matrix $A=(A_{w_i,w_j})_{r\times r}$. Consider the $r-$ tuples $\overline{w}=(w_1,\cdots,w_r)^T$ and $\overline{e}_w=(e_{w_1},\cdots, e_{w_r})^T.$ Then $\overline{e}_w=(I-A)\overline{w}.$ We have $A^i \overline{w}\subseteq (V_{2i},\cdots, V_{2i})^T$, hence $A^i \overline{w}\rightarrow 0$ as $i\rightarrow\infty.$ Now ,
$\overline{w}=\sum\limits_{i=0}^{\infty} A^{i}\overline{e_w}\in (J,\cdots, J)^{T},$ which proves the Lemma.

\end{proof}

\begin{corollary}
$\widehat{L}(\Gamma)$ is generated (as an ideal) by the set $\{e_w, e \in \bigcup \limits _{i} W_i \}$
\end{corollary}

 Let $V=S_1 \dot{\cup}...\dot{\cup} S_m$ be all connected components of the graph $(V, \tilde{\gamma}(V))$. Let $J_i$ be the closed ideal of $\widehat{L}(\Gamma)$ generated by the set $e_v, v \in S_i.$

 \begin{prop}
 \begin{itemize}
   \item[1.] $\widehat{L}(\Gamma)=J_1\bigoplus \cdots\bigoplus J_m;$
   \item[2.] each $J_i$ is a (topologically) simple algebra;
   \item[3.] $I(W_i)=\bigoplus J_i,$ the direct sum is taken over all $J$ such that $S_J \cap W_i \neq \emptyset .$
 \end{itemize}
  \end{prop}

   \begin{proof}
   The first assertion immediately follows from Lemma 21 and the corollary of Lemma 24. The second assertion follows from Lemma 23. The third assertion follows from Lemmas 21, 24, which finishes the proof of the Proposition.\\
   \end{proof}

   Remark that each component $S_i$ intersects just one minimal hereditary subset $W_i$. Indeed, if $S_i \cap W_i \backepsilon v$ and $S_i \cap W_j \backepsilon w,$ then by Lemma 20 the vertices $v$ and $w$ have a common descendant, which implies $i=j$. If $S_i \cap W_i =\emptyset$ for every $i$ then by Lemma 21(2) we have $J_i.id(e_v, v \in \bigcup \limits _{i} W_i )=(0).$ However, Lemma 24 implies that $id(e_v, v \in \bigcup \limits _{i} W_i)=I(W_1)\bigoplus\cdots\bigoplus I(W_k)=\widehat{L}(\Gamma), $ a contradiction.\\
   \bigskip
   Now we will show that an arbitrary finite connected graph has a specialization in which the decomposition of the Proposition 1(3) looks particularly nice.\\

   If $\gamma: V \rightarrow E$ is a specialization of a graph $\Gamma$ and $W$ is a hereditary subset of $V$ then the restriction of $\gamma$ to $W$ is a specialization
 of the graph $(W, E(W,W))$. We will denote this restriction as $\gamma_{\scriptstyle W}$\\
 \bigskip
 Let $W_1,\cdots ,W_k$ be the frame of the graph $\Gamma=(V,E)$. We call a specialization $\gamma:V \rightarrow E$ regular if \\
 (1). There are finitely many special paths with all vertices lying in $V\setminus(\bigcup \limits _{i}^k W_i)$,\\
 (2). Each graph $(W_i, \tilde{\gamma}_{\scriptstyle W_i})$ is connected,  $1 \leq i \leq k.$

  \begin{lemma}
   An arbitrary finite graph $\Gamma$ has a regular specialization.
 \end{lemma}

   By the proof of Lemma 13 a arbitrary specializations of non-sink minimal hereditary subsets $\gamma_i:W_i\rightarrow E(W_i, W_i)$ can be extended to a specialization $\gamma:V \rightarrow E$ with the property (1). Hence, it remains to find regular specializations on graphs $(W_i, E(W_i, W_i)),$ where $W_i$ does not consist of one sink.
    We have already mentioned that each graph $(W_i,E(W_i,W_i))$  is strongly connected, that is every vertex of it is a descendant of every other vertex.\\

 A graph $(V,E)$ is called a tree if there exists a vertex $v_0 \in V$ such that an arbitrary vertex in $V$ can be connected to $v_o$ by a unique path. An arbitrary strongly connected graph $(V,E)$ has a spanning subtree $(V,E'), E'\subseteq E$ (see [BJG]). Let $(W_i, E_i)$ be a spanning subtree of the graph $(W_i, E(W_i, W_i))$ ,
 $E_i \subset E(W_i,W_i)$. Let $w_i \in W_i$ be a such a vertex that an arbitrary vertex in $W_i$ can be connected to $w_i$ by a unique path in $(W_i, E_i)$.\\
 \bigskip
 If $w \in W_i, w \neq w_i$ then there exists a unique edge $e \in E_i$ such that $s(e)=w.$ We let $\gamma_i(w)=e.$ The edge $\gamma_i(w_i)$ is chosen arbitrarily in $s^{-1}(w_i).$ It is easy to see that the  graph $(W_i,\tilde{\gamma}_i)$ is connected which finishes the proof of the Lemma.\\
 Now the Proposition 1 implies

 \begin{prop}
  If $\gamma$ is a regular specialization then each ideal $I(W_i)$ is a topological graded simple algebra.
  \end{prop}

\begin{corollary}
Let $L(\Gamma)$ be a prime Leavitt path algebra. Let $\gamma$ be a regular specialization on $\Gamma$. Then $\widehat{L}(\Gamma)$ is topological graded simple.
\end{corollary}

\section*{Acknowledgement}
This project was funded by the Deanship of Scientific Research (DSR), King Abdulaziz University, under Grant No.
(27-130-36-HiCi). The authors, therefore, acknowledge technical and financial support of KAU.

%
%
%
%
%
%
%
%
%
%
\end{document}